\documentclass{amsart}
\usepackage{cases}
\usepackage{graphicx}
\usepackage{amssymb}
\usepackage{amsfonts}
\usepackage{amsmath}
\usepackage{amscd}
\usepackage{xcolor}
\usepackage[all]{xy}
\usepackage{tikz-cd}

 \newtheorem{thm}{Theorem}[section]
 \newtheorem{cor}[thm]{Corollary}
 \newtheorem{lem}[thm]{Lemma}
 \newtheorem{prop}[thm]{Proposition}
 \theoremstyle{definition}
 \newtheorem{defn}[thm]{Definition}

 \theoremstyle{remark}

 \numberwithin{equation}{section}

 \newcommand{\R}{\mathbb{R}}
 
 \newcommand{\Z}{\mathbb{Z}}
 \newcommand{\C}{\mathbb{C}}

 \newcommand{\F}{\mathcal{F}}

 \newcommand{\Zc}{\mathcal{Z}}
 \newcommand{\Rc}{\mathcal{R}}
 
 \newcommand{\E}{\mathcal{E}}

 \newcommand{\Mf}{\mathfrak{M}}

\begin{document}
\title[Equivariant Bordism]
 {Equivariant Bordism of 2-Torus Manifolds and Unitary Toric Manifolds}
\author{ Bo Chen, Zhi L\"u, Qiangbo Tan}

\address{School of Mathematics and Statistics, Huazhong University of Science and Technology, Wuhan, China} 
\email{bobchen@hust.edu.cn}
\address{School of Mathematical Science, Fudan University, Shanghai, China} 
\email{zlu@fudan.edu.cn}
\address{College of Science, Wuhan University of Science and Technology, Wuhan, China}
\email{tanqb@wust.edu.cn}
\thanks{This work is supported by grants from NSFC11801379}


\begin{abstract}
The equivariant bordism classification of manifolds with group actions is an essential subject in the study of transformation groups. We are interesting in the action of 2-torus group $\mathbb{Z}_2^n$ and torus group $T^n$, and study the equivariant bordism of 2-torus manifolds and unitary toric manifolds. In this paper, we give a new description of the group $\mathcal{Z}_n(\mathbb{Z}_2^n)$ of 2-torus manifolds, and determine the dimention of $\mathcal{Z}_n(\mathbb{Z}_2^n)$ as a $\mathbb{Z}_2$-vector space. With the help of toric topology, L\"u and Tan proved that the bordism groups $\mathcal{Z}_n(\Z_2^n)$ are generated by small covers. We will give a new proof to this result. These results can be generalized to the equivariant bordism of unitary toric manifolds, that is, we will give a new description of the group $\mathcal{Z}_n^U(T^n)$ of unitary torus manifolds, and prove that $\mathcal{Z}_n^U(T^n)$ can be generated by quasitoric manifolds with omniorientations.
\end{abstract}

\maketitle

\section{Introduction}

An $n$-dimensional {\em 2-torus manifold} is a smooth closed $n$-dimensional manifold equipped with an effective smooth $\Z_2^n$-action, so its fixed point set is empty or consists of isolated points (see \cite{L1,L2}). A {\em unitary $T^n$-manifold} is an oriented closed smooth manifold $M^m$  with an effective $T^n$-action such that its tangent bundle admits a $T^n$-equivariant stable complex structure, where $T^n = (S^1)^n$ is the compact torus group of dimention $n$. A unitary $T^n$-manifold $M^{2n}$ of dimension 2n with nonempty fixed set is called a {\em unitary torus manifold} or {\em unitary toric manifold} (see \cite{Masuda,HMasuda}). All the group actions in this article are assumed to be effective.

Davis and Januszkiewicz \cite{DJ} developed the theory of small covers and quasitoric manifolds, which are topological versions of real toric varieties and toric varieties, respectively. A small cover is a special case of 2-torus manifolds. A quasitoric manifold can be equipped with an omniorientation structure, this omniorientation determines a stable complex structure (see \cite{BPR1}). So a quasitoric manifold with an omniorientation is a special case of unitary toric manifolds.

In Section 2, we will focus on the unoriented equivarint bordism of 2-torus manifolds. In the 1960s Conner and Floyd \cite{CF1} begun the study of unoriented equivarint bordism theory of $\Z_2^n$ manifolds, they introduced a graded commutative algebra over $\Z_2$ with unit, $\Zc_*(\Z_2^n) = \bigoplus_{m\geq 0} \Zc_m(\Z_2^n)$, where the equivariant bordism group $\Zc_m(\Z_2^n)$ is consists of $\Z_2^n$-equivariant unoriented bordism classes of all $m$-dimensional 2-torus manifolds with effective action and fixing finite set. The addition and multiplication on $\Zc_*(\Z_2^n)$ are defined to be the disjoint union and the cartesian product with diagonal $\Z_2^n$-actions, respectively. For $n\leq 2$, Conner and Floyd showed that $\Zc_*(\Z_2) \cong Z_2$ and $\Zc_*(\Z_2^2) \cong \Z_2[u]$, where $u$ denotes the class of $\R P^2$ with the standard $\Z_2^2$-action. The equivariant unoriented bordism group of all $n$-dimensional 2-torus manifolds is $\Zc_n(\Z_2^n)$, which was also denoted by $\Mf_n$ in \cite{L1}. The group $\Zc_n(\Z_2^n)$ was determined for $n=3$ by L\"u \cite{L1} and for general $n\geq 3$ by L\"u and Tan \cite{LT1}. As showed in \cite{LT1}, $\Zc_n(\Z_2^n)$ is generated by all small covers of dimention $n$, and each class of $\Zc_n(\Z_2^n)$ contains a small cover as its representative. A computer algorithm was given there to search for the generators of $\Zc_n(\Z_2^n)$ and calculate the dimension $\text{dim}_{\Z_2}(\Zc_n(\Z_2^n))$. For $n>4$, the datas are too complex for the algorithm to work out. How to calculate $\text{dim}_{\Z_2}(\Zc_n(\Z_2^n))$ remains a problem.

As stated by Davis and Januszkiewicz in \cite{DJ}, there is a universal complex $X(\Z^n)$ such that every small cover can be obtained as a pullback from $X(\Z_2^n)$ via a suitable map. To prove that $\Zc_n(\Z_2^n)$ is generated by all small covers, L\"u and Tan \cite{LT1} defined a boundary operator $d_*$ on the dual of the Conner–Floyd algebra $\Rc_*(\Z_2^n)$, where $\Rc_*(\Z_2^n) = \bigoplus_{m\geq 0} \Rc_m(\Z_2^n)$ is the graded polynomial algebra over $\Z_2$ generated by all irreducible $\Z_2^n$-representations. Then the existence theorem of tom Dieck is equivalent to the condition $d_*f^*=0$ (\cite{LT1} Theorem 7.2). The differential $d_*$ is almost the same as the differential $\partial_*$ on the chain complex $(C_*(X(\Z_2^n)),\partial_*)$ with coefficient $\Z_2$, and $f^*$ corresponds to a closed chain, which represents an homology element in $H_{n-1}(X(\Z_2^n);\Z_2)$. With these observations, we will show the following main theorem for the equivariant bordism of 2-torus manifolds.
\begin{thm}\label{THM1-1}
    For each $n\geq 1$, the equivariant bordism group $\Zc_n(\Z_2^n)$ of 2-torus manifolds is isomorphic to the homology group  $H_{n-1}(X(\Z_2^n);\Z_2)$ of the universal $\Z_2^n$-complex.
    $$\Zc_n(\Z_2^n) \cong H_{n-1}(X(\Z_2^n);\Z_2)$$
\end{thm}

It has already been shown in \cite{BGV} that the homotopy type of $X(\Z_2^n)$ is the wedge of $A_n$ spheres $S^{n-1}$, where $A_n$ is a number related to the dimension $n$. So as a corollary, the dimention $\text{dim}_{\Z_2}(\Zc_n(\Z_2^n)) = A_n$ can be given straightforward, this solve the problem left in \cite{LT1}. 

\begin{cor}\label{COR1-1}
    The dimention of $\Zc_n(\Z_2^n)$ as a linear space over $\Z_2$ is
     $$\dim_{\Z_2}(\Zc_n(\Z_2^n)) = A_n = (-1)^n + \sum_{i=0}^{n-1}(-1)^{n-1-i}\frac{(2^{n}-2^i)\cdots(2^n-2^0)}{(i+1)!}$$    
\end{cor}

The homotopy type of $X(\Z_2^n)$ and the isomorphism Theorem \ref{THM1-1} inspire a new proof to a main theorem in \cite{LT1}.

\begin{cor}[\cite{LT1}, Theorem 2.5]\label{COR2-1}
    The group $\Zc_n(\Z_2^n)$ is generated by equivariant unoriented bordism classes of small covers of dimention $n$. 
\end{cor}

In Section 3, we will deal with the the equivariant bordism of unitary toric manifolds. There have been many different ways to study the equivariant unitary bordism (see \cite{BPR2,CF2,tD1,Lof1,Novikov1,Novikov2,Hanke}). However, the explicit ring structure is still difficult to calculate. Let $\Zc_{*}^U(T^n) = \bigoplus_{m\geq 0}\Zc_{m}^U(T^n)$ denote the equivariant unitary bordism ring of unitary $T^n$-manifolds with finite fixed point set, then $\Zc_{n}^U(T^n)$ is exactly the equivariant unitary bordism group of unitary toric manifolds. By analogy with the 2-torus manifolds, Darby \cite{Darby} generalized the methods in \cite{LT1} to the study of $\Zc_{n}^U(T^n)$. To deal with the effect of orientations and stably complex structures of the unitary manifolds, Darby considered the exterior algebra $\Lambda_{\Z}^*(J_n^{\C})$ which is generated by the exterior product of irreducible $T^n$-representations. The boundary operator $d_*$ can be defined on the dual of $\Lambda_{\Z}^*(J_n^{\C})$, and a similar existence theorem (\cite{Darby} Theorem 8.5) with the condition $d_*f^*=0$ still holds. But the unitary case turns out to be more difficult since the existence of an infinite number of irreducible representations. For $n=1$ and 2, the group $\Zc_{n}^U(T^n)$ was shown to be generated by quasitoric manifolds with omniorientations, and each class of $\Zc_{n}^U(T^n)$ contains a quasitoric manifold with omniorientation as its representative. But for $n>2$, it can only be conjectured that each class of $\Zc_{n}^U(T^n)$ contains a quasitoric manifold with omniorientation as its representative (see \cite{Darby}).

There exists a universal complex $X(\Z^n)$ characterizing the quasitoric manifolds with omniorientations (\cite{BGV} Proposition 2.1). As in the 2-torus case, the differential $d_*$ on the dual of $\Lambda_{\Z}^*(J_n^{\C})$ corresponds to the differential $\partial_*$ on the chain complex $(C_*(X(\Z^n))$ with coefficient $\Z$, and $f^*$ corresponds to a closed chain, which represents an homology element in $H_{n-1}(X(\Z^n);\Z)$. With these observations, we will show the following main theorem for equivariant bordism of unitary toric manifolds.

\begin{thm}\label{THM2-1}
    For each $n\geq 1$, the equivariant bordism group $\Zc_{2n}^U(T^n)$ of unitary torus manifolds is isomorphic to the homology group  $H_{n-1}(X(\Z^n);\Z)$ of the universal $\Z^n$-complex $X(\Z^n)$.
    $$\Zc_{2n}^U(T^n) \cong H_{n-1}(X(\Z^n);\Z)$$
\end{thm}

The homotopy type of $X(\Z^n)$ was proved to be the wedge of countably infinite spheres $S^{n-1}$ (see \cite{BGV}). So for $n\geq 2$, we have $\Zc_2^U(T^1) \cong \Z$ and $\Zc_{2n}^U(T^n)$ is isomorphic to a countably infinite direct sum of $\Z$ (Corollary \ref{COR3}). The homotopy type of $X(\Z^n)$ and Theorem \ref{THM2-1} inspire that $\Zc_{2n}^U(T^n)$ is generated by quasitoric manifold with omniorientations (see Proposition \ref{PROP1}). The generalized connect sums introduced by Buchstaber-Ray \cite{BPR1} and Darby \cite{Darby} suggest that the sum of two class $[M_1]$ and $[M_2]$ can still be represented by a quasitoric manifold with omniorientations, if $[M_1]$ and $[M_2]$ are represented by quasitoric manifolds with omniorientations. So we can prove the following proposition and solve the conjecture that $\Zc_n^U(T^n)$ is generated by quasitoric manifolds with omniorientations . 

\begin{prop}\label{PROP2-1}
    Each class of $\Zc_{2n}^U(T^n)$ contains an omnioriented quasitoric manifold as its representative.
\end{prop}


\section{Equivariant Bordism of 2-Torus Manifolds}
Following \cite{CF1}, let $\Rc_m(\Z_2^n)$ be the linear space over $\Z_2$, generated by the isomorphism classes of all $m$-dimensional real $\Z_2^n$-representations. The direct sum of representations defines a multiplication on  $\Rc_*(\Z_2^n) = \bigoplus_{m\geq 0} \Rc_m(\Z_2^n)$ and makes it into a graded commutative algebra over $\Z_2$, with unit the representation of degree 0. $\Rc_*(\Z_2^n)$ is called the {\em Conner–Floyd unoriented $\Z_2^n$-representation algebra}. As pointed out in \cite{L1}, $\Rc_*(\Z_2^n)$ is not the Grothendieck ring of $\Z_2^n$-representations. Every nontrivial irreducible real $\Z_2^n$-representation is 1-dimensional and  has the form $\tau_\rho: \Z_2^n\times \R \to \R $ with $\tau_\rho(g,x) = (-1)^{\rho(g)}$ for $\rho\in \text{Hom}(\Z_2^n,\Z_2)$, and $\tau_\rho$ is trivial if $\rho(g) = 0$ for all $g\in \Z_2^n$. Let $J_n^{\R} = \text{Hom}(\Z_2^n,\Z_2)\setminus \{0\}$ and regard it as the set of all nontrivial irreducible real $\Z_2^n$-representations, let $\Z_2[J_n^{\R}]$ denote the free polynomial algebra on $J_n^{\R}$ over $\Z_2$, then $\Rc_*(\Z_2^n)$ can be identified with $\Z_2[J_n^{\R}]$.

The essential relation between the equivariant unoriented bordism ring $\Zc_*(\Z_2^n) = \bigoplus_{m\geq 0} \Zc_m(\Z_2^n)$ and the Conner–Floyd unoriented $\Z_2^n$-representation algebra $\Rc_*(\Z_2^n)$ is characterized by the Stong \cite{Stong1}.

\begin{thm}[Stong]\label{Thm-Stong}
    The ring homomorphism 
        $$\varphi_*: \Zc_*(\Z_2^n) \to \Rc_*(\Z_2^n)$$
    defined by $[M] \mapsto \sum_{p\in M^{\Z_2^n}}[\tau_pM]$ is a monomorphism as algebras over $\Z_2$, where $\tau_pM$ denotes the real $\Z_2^n$-representation on the tangent space at the fixed point $p\in M^{\Z_2^n}$.
\end{thm}

For each fixed point $p\in M^{\Z_2^n}$, the tangent $\Z_2^n$-representation $\tau_pM$ can be decomposed into a direct sum of 1-dimensional irreducible real representations $\tau_pM=\oplus_{j=1}^n \tau_{i,j}$, the set $\{[\tau_pM]\in \Rc_*(\Z_2^n)|p\in M^{\Z_2^n}\}$ is called the {\em fixed point data} of $M$. By Stong's theorem, the equivariant bordism class $[M]$ in $\Zc_m(\Z_2^n)$ is determined by its fixed point data. But which polynomials in $\Rc_*(\Z_2^n)$ are polynomial from fixed point datas? In \cite{tD1}, tom Dieck proved an existence theorem, showed that the elements in $\Rc_*(\Z_2^n)$ are determined by their equivariant characteristic numbers. In \cite{LT1}, L\"u and Tan gave a simple proof to this theorem, and formulated it into the following result in terms of Kosniowski and Stong’s localization formula.

\begin{thm}[tom Dieck-Kosniowski-Stong]\label{Thm-tD}
     Let $\{\tau_1,\dots,\tau_l\}$ be a collection of $m$-dimensional faithful $\Z_2^n$-representations in $\Rc_m(\Z_2^n)$. Then a necessary and sufficient condition that $\tau_1+\dots+\tau_l\in Im(\varphi_m)$ (or $\{\tau_1,\dots,\tau_l\}$ is the fixed point data of a $\Z_2^n$-manifold $M^m$) is that for all symmetric polynomial functions $f(x_1,\dots,x_m)$ over $\Z_2$,
        \begin{equation}\label{Equ-tD}
            \sum_{i=1}^l\frac{f(\tau_i)}{\chi^{\Z_2^n}(\tau_i)}\in H^*(B\Z_2^n;\Z_2) 
        \end{equation}

     \noindent where $\chi^{\Z_2^n}(\tau_i)$ denotes the equivariant Euler class of $\tau_i$, which is a product of $m$ nonzero elements of $H^1(B\Z_2^n; \Z_2)$, and $f(\tau_i)$ means that variables $x_1, \dots, x_m$ in the function $f(x_1, \dots, x_m)$ are replaced by those $m$ degree-one factors in $\chi^{\Z_2^n}(\tau_i)$.
\end{thm}

If ${\tau_1,\dots,\tau_l}$ is the fixed data of a $\Z_2^n$-manifold $M^m$, then the polynomial \ref{Equ-tD} is exactly an equivariant Stiefel–Whitney number of $M^m$. Although all elements of $Im(\varphi_*)$ can be characterized by the above localization formula, the ring structure of $\Zc_*(\Z_2^n)$ is still far from settled. It was showed by Conner and Floyd in [12] that $\Zc_*(\Z_2) \cong \Z_2$, and $\Zc_*(\Z_2^2) \cong \Z_2[u]$, where $u$ denotes the class of $\R P^2$ with the standard $(\Z_2)^2$-action. But for $n > 2$, the structures and generators of $\Zc_*(\Z_2^n)$ are still unknown. 

When $m=n$, the group $\Zc_n(\Z_2^n)$ is exactly formed by the classes of all $n$-dimensional 2-torus manifolds, and the graded ring is denoted by $\Zc_*^\R = \bigoplus\Zc_n(\Z_2^n)$. The restriction of $\varphi_*|_{\Zc_*^\R}: \Zc_*^\R \to \Rc_*(\Z_2^n)$ is still a monomorphism. Suppose $[M^m]\in  \Zc_m(\Z_2^n)$ and $p\in M^{\Z_2^n}$ is a fixed point, then the set $\{\tau_{p,1},\dots, \tau_{p,m}\}$ contains a basis of $J_n^{\R}$. If $m=n$ and $M^n$ is a 2-torus manifold, then the monomial $\tau_p = \tau_{p,1}\cdots\tau_{p,n}$ corresponds to a faithful representation, which means that $\{\tau_{p,1},\dots,\tau_{p,n}\}$ is a basis of $J_n^{\R}$.

\begin{defn}
A homogeneous polynomial $g = \sum_i \tau_{i,1}\cdots\tau_{i,m}$ of degree $m\leq n$ in $\Z_2[J_n^{\R}]$ is called an {\em essential $\Z_2^n$-polynomial} if for each $i$, the elements in $\{\tau_{i,1},\dots,\tau_{i,m}\}$ are linearly independent in $\text{Hom}(\Z_2^n,\Z_2)$. An essential $\Z_2^n$-polynomial of degree $n$ is also called {\em faithful}. 
\end{defn}


For a 2-torus manifold $M$, the image $\varphi_*([M]) = \sum_{p\in M^{\Z_2^n}}[\tau_pM]\in \Z_2[J_n^{\R}]$ is a faithful polynomial. The space $J_n^{\R}$ has a dual space $J_n^{*\R} = \text{Hom}(\Z_2, \Z_2^n)\setminus \{0\}$, and they are both isomorphic to $\Z_2^n$, the is defined by the following composition of homomorphisms:
$$<\cdot,\cdot>: J_n^{*\R}\times J_n^{\R} \to Hom(\Z_2,\Z_2)\cong \Z_2$$
$$<\xi,\rho>=\rho\circ \xi$$
Similarly, we can define essential and faithful $\Z_2^n$-polynomials in $\Z_2[J_n^{*\R}]$. If $g = \sum_i \tau_{i,1}\dots \tau_{i,n}$ is a faithful $\Z_2^n$-polynomial in $\Z_2[J_n^{\R}]$, then each basis $\{\tau_{i,1},\cdots, \tau_{i,n}\}$ has a dual basis $\{s_{i,1},\cdots, s_{i,n}\}$, and they determine a unique homogeneous polynomial $g^* = \sum_i s_{i,1}\dots s_{i,n}$ in $\Z_2[J_n^{*\R}]$, which is called the {\em dual $\Z_2^n$-polynomial} of $g$. It is obvious that $g$ is faithful in $\Z_2[J_n^{\R}]$ if and only if $g^*$ is faithful in $\Z_2[J_n^{*\R}]$.

In \cite{LT1}, L\"u and Tan defined a differential operator $d$ on the graded ring $\Z_2[J_n^{*\R}]$ as follows: for each monomial $\rho_1\dots\rho_k \in \Z_2[J_n^{*\R}]$ with $k\geq 1$, let
    \begin{equation}\label{Diff-1}
        d_k(\rho_1\dots \rho_k) = 
        \begin{cases}
        \sum_{i=1}^k \rho_1\dots \rho_{i-1} \hat{\rho_{i}} \rho_{i+1}\dots \rho_{k}  & \text{if}\  k>1 \vspace{2mm} \\  
        1 & \text{if}\ k=1
        \end{cases}
    \end{equation} 
and $d_0(1) = 0$, where the symbol $\hat{\rho_{i}}$ means that $\rho_i$ is deleted. It is easy to check that $d^2=0$, so it can be regarded as a boundary operator. For 2-torus manifolds, it was proved in \cite{LT1} that the existence Theorem \ref{Thm-tD} of tom Dieck is equivalent to the following.

\begin{thm}[L\"u-Tan]\label{Thm-LT}
 Let $g = \sum_i t_{i,1}\dots t_{i,n}$ be a faithful $\Z_2^n$-polynomial in $\Z_2[J_n^{\R}]$. Then $g\in Im\varphi_n$ if and only if $d_n(g^*) = 0$.
\end{thm}

Let $0\leq m\leq n-1$, let $\E_m(J_n^{*\R})$ be the set of all essential $\Z_2^n$-polynomials of degree $m+1$ in $\Z_2[J_n^{*\R}]$. If we regard an element in $J_n^{*\R}$ as a vertex, and an essential monomial $\tau_{i,1}\cdots\tau_{i,m}$ as a simplex of dimention $m-1$, then the graded polynomial ring $\E_*(J_n^{*\R}) = \bigoplus_{0\leq m\leq n-1} \E_m(J_n^{*\R})$ can be regarded as a chian complex with chain map $d_*$. This point of view is similar to the definition of {\em universal $\Z_2^n$-space} $X(\Z_2^n)$, which was introduced and used to characterized the small covers by Davis and Januszkiewicz in \cite{DJ}. 

\begin{defn}
    A subset $\{v_1,\dots,v_m\}$ of  $\Z_2^n$ is  {\em unimodular} if span$\{v_1,\dots,v_m\}$ is a direct summand of $\Z_2^n$ with dimension $m$. Note that a subset of a unimodular set is itself unimodular. The collection of all unimodular subsets of $\Z_2^n$ forms a simplicial complex $X(\Z_2^n)$ on unimodular vertices, which is called the {\em universal $\Z_2^n$-complex}.
\end{defn}

As showed in \cite{BGV}, the universal complex $X(\Z_2^n)$ is a pure simplicial complex and an $(n-1)$-connected Cohen-Macaulay space. It is the universal simplicial complex classifying small covers of dimention $n$. Actually, it can be used to classify the equivariant bordism of 2-torus manifolds as stated in Theorem \ref{THM1-1} and we now give a proof to that theorem.


\begin{proof}[{\bf Proof of Theorem~\ref{THM1-1}}]
    Let $0\leq m\leq n-1$. Let $r: J_n^{*\R}= Hom(\Z_2, \Z_2^n) \to \Z_2^n$ be the standard isomorphim. For an essential monomial $s_i^* = s_{i,1}\cdots s_{i,m} \in \E_{m}(J_n^{*\R})$,  the set $\{r(s_{i,1}),\dots,r(s_{i,m})\}$ is unimodular in $\Z_2^n$, it spans a simplex of dimention $m-1$ in the universal complex $X(\Z_2^n)$, denoted by $r(s_i^*)$. Let $(C_*(X(\Z_2^n)),\partial_*) = \sum_{0\leq m\leq n-1} (C_m(X(\Z_2^n),\partial_m)$ be the chain complex of $X(\Z_2^n)$, then the map $r$ induces a homomorphism 

$$r_m: \E_m(J_n^{*\R}) \to C_{m}(X(\Z_2^n)) $$

\noindent which maps an essential polynomial  $g^* = \sum s_i^*\in \E_m(J_n^{*\R})$ to a chain $r_m(g^*) = \sum r(s_i^*) \in C_m(X(\Z_2^n))$. It is easy to see that $r$ is an isomorphism of groups and the following is commutative:

$$
\xymatrix{
   \E_m(J_n^{*\R}) \ar[d]^{d_m} \ar[r]^{r_m}  & C_m(X(\Z_2^n)) \ar[d]^{\partial_{m}}            \\
   \E_{m-1}(J_n^{*\R}) \ar[r]^{r_{m-1}} & C_{m-1}(X(\Z_2^n))                       }
$$

\noindent where $d_m$ is the differential (\ref{Diff-1}) defined on $\Z_2[J_n^{*\R}]$, and $\partial_{m}$ is the differential on the chain complex $C_m(X(\Z_2^n))$. Regard $(\E_*(J_n^{*\R}), d)$ as a chian complex, then the above homomorphism $r$ is a chain isomorphism, it induces an isomorphism of homology
$$r_{m*}: H_m(\E_*(J_n^{*\R})) \cong H_m(X(\Z_2^n))$$ 

Let $g$ be a faithful $\Z_2^n$-polynomial in $\Z_2[J_n^{\R}]$, then $g^*$ is a faithful $\Z_2^n$-polynomial in $\Z_2[J_n^{*\R}]$. By Theorem \ref{Thm-Stong}, the equivariant bordism group $\Zc_n(\Z_2^n)$ of 2-torus manifolds is isomorphic to the image $Im\varphi_n$. By Theorem \ref{Thm-LT}, $g\in Im\varphi_n$ if and only if $d_n(g^*) = 0$, that is, $g^*$ is a closed chain in $\E_*(J_n^{*\R})$, it represents a homology class $[g^*]\in H_{n-1}(\E_*(J_n^{*\R}))\cong H_{n-1}(X(\Z_2^n))$, so we have
$$\Zc_n(\Z_2^n) \cong Im\varphi_n \cong H_{n-1}(\E_*(J_n^{*\R});\Z_2) \cong H_{n-1}(X(\Z_2^n);\Z_2)$$
\end{proof}

The homotopy of the universal complex $X(\Z_2^n)$ was determined in \cite{BGV}, we recall their result to calculate the dimension of $\Zc_n(\Z_2^n)$, the proof is straightforward.

\begin{thm}\cite{BGV}\label{Thm-BGV}
The simplicial complex $X(\Z_2^n)$ is homotopy equivalent to the wedge of $A_n$ spheres $S^{n-1}$, where 
 $$A_n = (-1)^n + \sum_{i=0}^{n-1}(-1)^{n-1-i}\frac{(2^{n-1}-2^i)\cdots(2^n-2^0)}{(i+1)!}$$  
\end{thm}

\begin{proof}[{\bf Proof of Corollary~\ref{COR1-1}}]
    The proof is straightforward by Theorem \ref{THM1-1} and Theorem \ref{Thm-BGV}.
\end{proof}


It was proved in \cite{LT1} that the generators of $\Zc_n(\Z_2^n)$ can be chosen to be small covers over products of simplices. Small covers are important models for $\Z_2^n$-manifolds, which are defined and studied by Davis and Januszkiewicz in \cite{DJ}, they are topological versions of real toric varieties. We recall their definitions and basic properties.

\begin{defn}
    An $n$-dimensional {\em small cover} $\pi: M^{n} \to P^n$ is a smooth closed $n$-dimensional manifold $M^n$ with a locally standard $\Z_2^n$-action such that its orbit space is homeomorphic to a $n$-dimensional simple convex polytope $P^n$, here locally standard means that the action is locally isomorphic to a faithful representation of $\Z_2^n$ on $\R^n$. 
\end{defn}

Each small cover $\pi: M^n\to P^n$ determines a characteristic function $\lambda:\F(P^n)\to J_n^{*\R}$, where $\F(P^n)$ consists of all facets (i.e., $(n-1)$-dimensional faces) of $P^n$. Given a vertex $v$ of $P^n$, since $P$ is simple, there are $n$ facets $F_1, \dots, F_n$ in $\F(P^n)$ such that $v = F_1 \cap\dots\cap F_n$, characteristic function has the property that the elements $\{\lambda(F_1),\dots,\lambda(F_n)\}$ are linearly independent, so their product $\lambda_v = \lambda(F_1)\cdots\lambda(F_n)$ is essential, and the polynomial $f^*_{\lambda}(P) = \sum_{v\in V(P)}\lambda_v$ is a faithful polynomial of degree $n$ in $\Z_2[J_n^{*\R}]$. We call the map $\lambda$ a $J_n^{*\R}$-coloring of $P^n$ here, and the $f^*_{\lambda}(P)$ is called the {\em $J_n^{*\R}$-coloring polynomial } of $(P^n,\lambda)$.

For a small cover $\pi: M^{n} \to P^n$, the image $f(P^n,\lambda) = \varphi_n([M])$ under the Stong homomorphism is a faithful polynomial in $\Z_2[J_n^{\R}]$, and its dual is just the $J_n^{*\R}$-coloring polynomial $f^*_{\lambda}(P)$. We recall the result about the generators of $\Zc_n(\Z_2^n)$, and give a new proof to it.

\begin{proof}[{\bf Proof of Corollary \ref{COR2-1}}]
    By Theorem \ref{Thm-BGV}, the universal complex $X(\Z_2^n)$ is homotopy to the wedge of $A_n$ spheres $S^{n-1}$, we call such spheres {\em basic spheres} of $X(\Z_2^n)$. By Theorem \ref{THM1-1}, each equivariant bordism class of $\Zc_n(\Z_2^n)$ is corresponding to an element in $H_{n-1}(X(\Z_2^n);\Z_2)$, which are generalized by the homology class of basic spheres. Since $X(\Z_2^n)$ is pure, the dual of a basic sphere $S^{n-1}$ is the boundary $\partial{P^n}$ of a simple polytope $P^n$, and the unimodular subsets of $S^{n-1}$ induce a $J_n^{*\R}$-coloring $\lambda$ on $P^n$. By Davis–Januszkiewicz's theory of small covers, the pair $(P^n,\lambda)$ determines a unique small cover $\pi:M^n\to P^n$ such that $\lambda$ is the characteristic function. So the equivariant bordism class corresponding to the baisc sphere $S^{n-1}$ is represented by a small cover $M$. 
\end{proof}

\section{Equivariant Bordism of Unitary Torus Manifolds}
A {\em unitary manifold} $M$ is a smooth manifold endowed with a complex structure on the stable tangent bundle of $M$. If the complex structure is given on the tangent bundle $TM$ of $M$, M is called an {\em almost complex manifold}. Let $G$ be a compact Lie group, if the action of $G$ preserves the tangential stably complex structure of unitary manifold $M$, then $M$ is called a {\em unitary $G$-manifold}.

Let $\Omega^{U,G}$ denote the equivariant unitary bordism ring of unitary $G$-manifolds. There have been many different ways to study the equivariant unitary bordism, however, the explicit ring structure of $\Omega^{U,G}$ is still difficult to calculate, only partial results are known. In \cite{tD1}, tom Dieck gave a commutative diagram
$$
\begin{tikzcd}[column sep=scriptsize, row sep=scriptsize]
     \Omega_*^{U,G} \arrow[r, "\varphi_{\Omega}"] \arrow[d] & MU_*[e_V,e_V^{-1},Y_{V,d}] \arrow[d] \\
     MU_*^{G} \arrow[r]  & MU_*[e_V, e_V^{-1}, Y_{V,d}]
\end{tikzcd}
$$

\noindent where $MU_*^{T^n}$ is the homotopy theoretic equivariant unitary bordism ring, and $MU_*$ is the ordinary homotopy theoretic unitary bordism ring, $e_V$ are the Euler classes of nontrivial irreducible complex $T^n$-representations V, $Y_{V,d}$ are the classes of degree $2d$ $(2 \leq d \leq \infty)$ represented by the $T^n$-bundle $E\otimes V \to \C P^{d-1}$ with $E\to \C P^{d-1}$ being the hyperplane line bundle. When $G = T^n$ is the compact torus of rank $n$, Hanke \cite{Hanke} showed that the image $\text{Im}(\varphi_{\Omega})\subset MU_*[e_V^{-1},Y_{V,d}]$, and proved that the following

$$
\begin{tikzcd}[column sep=scriptsize, row sep=scriptsize]
     \Omega_*^{U,T} \arrow[r, "\varphi_{\Omega}"] \arrow[d] & MU_*[e_V^{-1},Y_{V,d}] \arrow[d] \\
     MU_*^{T} \arrow[r]  & MU_*[e_V, e_V^{-1}, Y_{V,d}]
\end{tikzcd}
$$

\noindent is a pullback square and  has all maps injective.

For a unitary $T^n$-manifold $M$, each component of the fixed point set $M^T$ is again a unitary $T$-manifold. If $M$ has an isolated fixed point, then $\text{dim}M$ is even and the tangent space $T_pM$ is a complex $T^n$-representation. Let $\Zc_{*}^U(T^n) = \bigoplus_{m\geq 0}\Zc_{m}^U(T^n)$ denote the subring of $\Omega_*^{U,T}$ consists of all classes that can be represented by unitary $T^n$-manifolds with finite fixed point set. In his Ph.D’s thesis \cite{Darby}, Darby refined Hanke's result to the following pull back square with all maps injective.

$$
\begin{tikzcd}[column sep=scriptsize, row sep=scriptsize]
     \Zc_{*}^U(T^n) \arrow[r, "\varphi_{\Omega}"] \arrow[d] & MU_*[e_V^{-1}] \arrow[d] \\
     MU_*^{T} \arrow[r]  & MU_*[e_V, e_V^{-1}, Y_{V,d}]
\end{tikzcd}
$$

Let $M^{2m}$ be a unitary $T^n$-manifold with finite fixed point set, for each isolated fixed point $p\in M^T$, there are two orientations on the tangent space $T_pM$: the one induced from the orientation of $M$ and the other induced from the complex structure on $T_pM$. Define the sign $\varepsilon(p)$ of $p$
    \begin{equation}\label{Equ-1}
        \varepsilon(p) = 
        \begin{cases}
        +1  & \text{if these two orientations coincide}\vspace{2mm} \\  
        -1 & \text{otherwise}
        \end{cases}
    \end{equation} 
The tangent space $T_pM$ is a complex $T^n$-representation and can be decomposed into a direct sum of irreducible representations $\tau_pM = \tau_{p,1}\oplus\dots\oplus\tau_{p,m}$. Let $J_n^{\C} = \text{Hom}(T^n,S^1)\setminus \{0\}$ and regard it as the set of all irreducible complex $T^n$-representations. Let $\Z[J_n^{\C}]$ be the polynomial algebras on $J_n^{\C}$ over $\Z$, each element in $J_n^{\C}$ is of degree 2. The set $\{[\tau_pM]\in\Z[J_n^\C]|p\in M^T\}$ is called {\em the fixed point data}. The homomorphism $\varphi_{\Omega}$ restrict to $\Zc_{*}^U(T^n)$ can be written as
$$\varphi_{\Omega}: \Zc_*^U(T^n) \to \Z[J_n^\C]$$
$$[M] \mapsto \sum_{p\in M^{T^n}}\varepsilon(p)\prod_{i=1}^m\tau_{p,i}$$

\noindent By Darby's pullback square above, the $\varphi_{\Omega}$ is injective. To compare with Stong's Theorem \ref{Thm-Stong} in the 2-torus case, we restate the Darby's result in the following theorem.

\begin{thm}[\cite{Darby}]\label{Thm-Darby}
    The ring homomorphism $\varphi_{\Omega}: \Zc_*^U(T^n) \to \Z[J_n^\C]$ is a monomorphism as algebras over $\Z$.
\end{thm}

Unitary toric manifolds was introduced by Masuda in \cite{Masuda}, which is a special case of unitary $T^n$-manifolds. The equivariant bordism group $\Zc_n^U(T^n)$ is formed by the classes of all $2n$-dimensional unitary toric manifolds, and the graded ring is $\Zc_*^\C = \bigoplus_{n\geq 0}\Zc_n(T^n)$. To eliminate the sign $\varepsilon(p)$ of a fixed point $p$ and absorb the effect of orientations in the image $\varphi_{\Omega}([M])$, consider the exterior algebra $\Lambda_{\Z}^*(J_n^{\C})$ as in Darby \cite{Darby}. 

\begin{defn}
    The {\em exterior algebra} $\Lambda_{\Z}^*(J_n^{\C})$ is the quotient algebra $\Lambda_{\Z}^*(J_n^{\C}): = \Z[J_n^\C]/I$, where $I$ is the ideal generated by all elements of the form $\tau^2$ and expressions of the form $\tau_1\tau_2+\tau_2\tau_1$. The exterior product is $\tau_1\wedge \tau_2 = \tau_1\tau_2 \text{ mod }I$, and $\Lambda_{\Z}^*(J_n^{\C}) =\bigoplus_{k\geq 0}\Lambda_{\Z}^k(J_n^{\C})$ is a graded algebra such that $\Lambda_\Z^k[J_n^{\C}]\wedge \Lambda_\Z^l[J_n^{\C}] \subset \Lambda_\Z^{k+l}[J_n^{\C}]$.
\end{defn}

Suppose $M^{2n}$ is a unitary torus manifold and $p\in M^T$ is a fixed point, the tangent $T^n$-representaion at $p$ can be decomposed as a product of irreducible $T^n$-representations $\tau_p = \tau_{p,1}\dots\tau_{p,n}$. By \cite{Masuda} Lemma 1.3 (1), the set $\{\tau_{p,1},\dots, \tau_{p,n}\}$ is a basis of $J_n^{\C}$. For each fixed poit $p$, rearrange the basis $\{\tau_{p,1},\dots, \tau_{p,n}\}$ if necessary such that 
$$\det[\tau_{p,1} \dots \tau_{p,n}] = \varepsilon(p)$$

\noindent where the $n\times n$ integral matrix $[\tau_{p,1} \dots \tau_{p,n}]$ is formed by taking the ith column to be the vector $\tau_{p,n}\in J_n^{\C}\cong \Z^n$. Then $\varphi_{\Omega}$ induces a homomorphism $\varphi_*^U: \Zc_*^\C \to \Lambda_{\Z}^*(J_n^{\C})$ such that 
$$\varphi_*^U([M]) =  \sum_{p\in M^{T^n}}\tau_{p,1}\wedge\dots\wedge\tau_{p,n}$$
By Theorem \ref{Thm-Darby}, the following result is straightforward.
\begin{cor}[\cite{Darby}]\label{Cor-Darby}
    The ring homomorphism $\varphi_*^U: \Zc_*^\C \to \Lambda_{\Z}^*(J_n^{\C})$ is a monomorphism as algebras over $\Z$.
\end{cor}

To describe the image of $\varphi_*^U$, we introduce the definitions of essential and faithful polynomials as in the 2-torus case.

\begin{defn}
An exterior polynomial $g = \sum_i \tau_{i,1}\wedge\cdots\wedge\tau_{i,m}$ of degree $2m\leq 2n$ in $\Z_2[J_n^{\C}]$ is called an {\em essential exterior $T^n$-polynomial} if for each $i$, the elements in $\{\tau_{i,1},\dots,\tau_{i,m}\}$ are linearly independent in $\text{Hom}(T^n,S^1)$. An essential exterior $T^n$-polynomial of degree $2n$ is also called {\em faithful}. 
\end{defn}

For a unitary toric manifold $M^{2n}$, the image $\varphi_*^U([M])$ is a faithful extorior polynomial. Let $J_n^{*\C} = \text{Hom}(S^1, T^n)$ be the dual space of $J_n^{\C}$, both $J_n^{\C}$ and $J_n^{*\C}$ are isomorphic to $\Z^n$, where the dual homomorphism is defined by a composition as in the 2-torus case. Similarly, we can define essential and faithful $T^n$-polynomials in $\Z_2[J_n^{*\C}]$. Each faithful polynomial $g$ in $\Lambda_{\Z}^*(J_n^{\C})$ has a dual $g^*$ which is faithful in $\Lambda_{\Z}^*(J_n^{*\C})$. A differential operator $d_*$ on $\Lambda_{\Z}^*(J_n^{*\C})$ was defined in \cite{Darby} as follows: for each monomial $\rho_1\wedge\dots\wedge\rho_k \in \Lambda[J_n^{*\C}]$ with $k\geq 1$, define
    \begin{equation}\label{Diff-2}
        d_k(\rho_1\dots \rho_k) = 
        \begin{cases}
        \sum_{i=1}^k (-1)^{i+1} \rho_1\dots \rho_{i-1} \hat{\rho_{i}} \rho_{i+1}\dots \rho_{k}  & \text{if}\  k>1 \vspace{2mm} \\  
        1 & \text{if}\ k=1
        \end{cases}
    \end{equation} 
and $d_0(1) = 0$, where the symbol $\hat{\rho_{i}}$ means that $\rho_i$ is deleted. It is easy to check that $d^2=0$. For unitary torus manifolds, the existence theorem similar to Theorem \ref{Thm-LT} can be stated as follows.

\begin{thm}[\cite{Darby} Theorem 8.5]\label{Thm-Darby2}
 Let $g = \sum_i t_{i,1}\wedge\dots\wedge t_{i,n}$ be a faithful extorior $T^n$-polynomial in $\Lambda_{\Z}^*(J_n^{\C})$. Then $g\in Im\varphi_n^U$ if and only if $d(g^*) = 0$.
\end{thm}

Let $0\leq m\leq n-1$, let $\E_m(J_n^{*\R})$ be the set of all essential extorior polynomials of degree $2(m+1)$ in $\Z[J_n^{*\C}]$. As in the 2-torus case, the graded polynomial ring $\E_*(J_n^{*\C}) = \bigoplus_{0\leq m\leq n-1} \E_m(J_n^{*\C})$ can be regarded as a chian complex with chain map $d_*$. This point of view is similar to the definition of {\em universal $\Z^n$-space} $X(\Z^n)$ (\cite{BGV}).

\begin{defn}
    A subset $\{v_1,\dots,v_m\}$ of  $\Z^n$ is  {\em unimodular} if span$\{v_1,\dots,v_m\}$ is a direct summand of $\Z^n$ with rank $m$. Note that a subset of a unimodular set is itself unimodular. The collection of all unimodular subsets of $\Z^n$ forms a simplicial complex $X(\Z^n)$ on unimodular vertices, which is called the {\em universal $\Z^n$-complex}.
\end{defn}

Note that this universal $\Z^n$-complex $X(\Z^n)$ is different from the universal complex characterizing the quasitoric manifolds by Davis and Januszkiewicz in \cite{DJ}. It is clear that the simplicial complexes $X(\Z^n)$ is $(n-1)$-dimensional and pure. As showed in \cite{BGV}, the universal complex $X(\Z_2^n)$ is $(n-1)$-connected Cohen-Macaulay. As in the 2-torus case, the universal simplicial complex $X(\Z^n)$ can be used to classify the equivariant bordism of unitary torus manifolds.

\begin{proof}[{\bf Proof of Theorem \ref{THM2-1}}]
    The proof here is just the unitary version of the proof of Theorem \ref{THM1-1}. Let $0\leq m\leq n-1$. Let $r: J_n^{*\C} = \text{Hom}(S^1, T^n) \to \Z^n$ be the standard isomorphim. For an essential exterior monomial $s_i^*= s_{i,1}\cdots s_{i,m} \in \E_{m}(J_n^{*\C})$, the set $\{r(s_{i,1}),\dots,r(s_{i,m})\}$ is unimodular in $\Z^n$, it spans a simplex of dimention $m-1$ in the universal complex $X(\Z^n)$, denoted by $r(s_i^*)$. Let $(C_*(X(\Z^n)),\partial_*) = \sum_{0\leq m\leq n-1} (C_m(X(\Z^n)),\partial_m)$ be the chain complex of $X(\Z^n)$, then the map $r$ induces a homomorphism 
    $$r_m: \E_m(J_n^{*\C}) \to C_{m}(X(\Z^n)) $$
    \noindent which maps an essential polynomial  $g^* = \sum s_i^*\in \E_m(J_n^{*\C})$ to a chain $r_m(g^*) = \sum r(s_i^*) \in C_m(X(\Z^n))$. It is easy to see that $r$ is an isomorphism of groups and the following is commutative:
    $$
    \xymatrix{
    \E_m(J_n^{*\C}) \ar[d]^{d_m} \ar[r]^{r_m}  & C_m(X(\Z^n)) \ar[d]^{\partial_{m}}            \\
    \E_{m-1}(J_n^{*\C}) \ar[r]^{r_{m-1}} & C_{m-1}(X(\Z^n))                       } 
    $$
    \noindent where $d_m$ is the differential (\ref{Diff-2}) defined on $\Lambda_{\Z}^m[J_n^{*\R}]$, and $\partial_{m}$ is the differential on the chain complex $C_m(X(\Z^n))$. Regard $(\E_*(J_n^{*\R}), d)$ as a chian complex, then the above homomorphism $r$ is a chain isomorphism, it induces an isomorphism of homology
    $$r_{m*}: H_m(\E_*(J_n^{*\R})) \cong H_m(X(\Z_2^n))$$ 

 Let $g$ be a faithful extorior polynomial in $\Z[J_n^{\C}]$, then $g^*$ is a faithful extorior polynomial in $\Z_2[J_n^{*\R}]$. By Corollary\ref{Cor-Darby}, the equivariant bordism group $\Zc_n^U(T^n)$ of unitary torus manifolds is isomorphic to the image $Im\varphi_n^U$. By Theorem \ref{Thm-Darby2}, $g\in Im\varphi_n^U$ if and only if $d(g^*) = 0$, that is, $g^*$ is a closed chain in $\E_*(J_n^{*\C})$, it represents a homology class $[g^*]\in H_{n-1}(\E_*(J_n^{*\C}))\cong H_{n-1}(X(\Z^n))$, so we have
    $$\Zc_n^U(T^n) \cong Im\varphi_n^U \cong H_{n-1}(\E_*(J_n^{*\C});\Z) \cong H_{n-1}(X(\Z^n);\Z)$$
\end{proof}

The homotopy of the universal complex $X(\Z^n)$ was determined in \cite{BGV}, we recall their result to calculate the equivariant bordism groups $\Zc_n^U(\Z^n)$, the proof is straightforward.

\begin{thm}[\cite{BGV} Theorem 1.3\label{Thm-BGV2}]
For $n\geq 2$, the simplicial complex $X(\Z^n)$ is homotopy equivalent to a countably infinite wedge of $(n-1)$-spheres.
\end{thm}

\begin{cor}\label{COR3}
    The group $\Zc_2^U(T^1) \cong \Z$ and $\Zc_{2n}^U(T^n)$ is isomorphic to a countably infinite direct sum of $\Z$ for $n\geq 2$.
\end{cor}

Quasitoric manifolds are important models for $T^n$-manifolds, which are defined and studied by Davis and Januszkiewicz in \cite{DJ}, they are topological versions of toric varieties. We recall their definitions and basic properties.

\begin{defn}
    An $2n$-dimensional {\em quasitoric manifold} $\pi: M^{2n} \to P^n$ is a smooth closed $2n$-dimensional manifold $M^{2n}$ with a locally standard $T^n$-action such that its orbit space is homeomorphic to a $n$-dimensional simple convex polytope $P^n$, here locally standard means that the action is locally isomorphic to a faithful representation of $T^n$ on $\C^n$. 
\end{defn}

A quasitoric manifold $\pi: M^{2n}\to P^n$ determines a characteristic function $\lambda:\F(P^n)\to J_n^{*\C}$, where $\F(P^n)$ is the set of all facets of $P^n$. For each facet $F\in\F(P^n)$, the preimage $\pi^{-1}(F)$ is submanifold of dimention $2(n-1)$, and itself is also a quasitoric manifold. Given a vertex $v = F_1 \cap\dots\cap F_n$ of $P^n$, where $F_1, \dots, F_n \in \F(P^n)$, then the elements $\lambda(F_1),\dots,\lambda(F_n)$ form a basis of $J_n^{*\C}$. Quasitoric manifolds can be equipped with an additional structure called {\em omniorientation}, providing a combinatorial description for unitary structures.

\begin{defn}
    An {\em omniorientation} of a quasitoric manifold $M^{2n}$ consists of a choice of an orientation for $M^{2n}$ and for every facial submanifold $M^{2(n-1)} = \pi^{-1}(F_i), i = 1,... ,m$.
\end{defn}

Note that a choice of orientation for $P^n$ is equivalent to a choice of orientation for $M^{2n}$, and an omniorientation on $M^{2n}$ is equivalent to a choice of orientations for $P^n$ and all its facets. As proved by Buchstaber and Ray in \cite{BPR1} Proposition 4.5, every omniorientation of a quasitoric manifold $M^{2n}$ determines a stably complex structure on it, which is compatible with the action of the torus. So a quasitoric manifold $M^{2n}$ with an omniorientation is a special case of unitary toric manifolds. If we forget the group action, then the bordism class $[M]$ is an element of the ordinary unitary bordism group $\Omega_{2n}^U$. For $n>1$, every complex cobordism class in $\Omega_{2n}^U$ contains a quasitoric manifold with an omniorientation (\cite{BPR1} Theorem 5.9). 

\begin{defn}
    A {\em quasitoric pair} $(P^n,\Lambda)$ consists of a combinatorial oriented simple $n$-polytope $P^n$ and an integral $(n \times m)$-matrix $\Lambda$ whose columns $\lambda_i$ satisfy 
    $$\det(\lambda_{i_1}\cdots\lambda_{i_n}) = \pm 1 \text{, whenever } v = F_{i_1} \cap\dots\cap F_{i_n} \text{is a vertex of } P^n $$
    where $F_{i_j}\in\F(P^n), j=1,\dots,n, $ are some facets of $P^n$. 
\end{defn}

Two quasitoric pairs $(P_1,\Lambda_1)$ and $(P_2,\Lambda_2)$ are equivalent if and only if $P_1 = P_2$ and there exists an $(m \times m)$-permutation matrix $\Sigma$ such that $\Lambda_1 = \Lambda_2 \Sigma$, the matrix $\Sigma$ may be thought as a reordering of the facets of $P^n$. The following was showed by Buchstaber, Panov and Ray in \cite{BPR1}, see also \cite{Darby} Corollary 5.28.

\begin{lem}[\cite{BPR1}\cite{Darby}]\label{Lem-BR}
    There is a 1-1 correspondence between equivalence classes of omnioriented quasitoric manifolds and combinatorial quasitoric pairs.    
\end{lem}


\begin{prop}\label{PROP1}
    The group $\Zc_{2n}^U(T^n)$ is generated by equivariant unitary bordism classes of $2n$-dimensional omnioriented quasitoric manifolds. 
\end{prop}

\begin{proof}
    By Theorem \ref{Thm-BGV2}, the universal complex $X(\Z^n)$ is homotopy to the wedge of countably infinite spheres $S^{n-1}$, we call such spheres {\em basic spheres} of $X(\Z^n)$. By Theorem \ref{THM2-1}, each equivariant bordism class of $\Zc_n(T^n)$ corresponds to an element in $H_{n-1}(X(\Z^n);\Z)$, which is generalized by the homology class of basic spheres. Since $X(\Z^n)$ is pure, the dual of a basic sphere $S^{n-1}$ is the boundary $\partial{P^n}$ of a simple polytope $P^n$, and the unimodular subsets of $S^{n-1}$ induce a $J_n^{*\R}$-coloring $\lambda$ of $P^n$. By Lemma \ref{Lem-BR}, the quasitoric pair $(P^n,\lambda)$ determines a unique omnioriented quasitoric manifold $\pi:M^{2n}\to P^n$ such that $\lambda$ is the characteristic function. So the equivariant bordism class corresponding to the baisc sphere $S^{n-1}$ is represented by an omnioriented quasitoric manifold. 
\end{proof}

\begin{proof}[{\bf Proof of Proposition \ref{PROP2-1}}]
   The case for $n=1$ is clear and was shown by \cite{Darby} Corollary 8.8. For $n>1$, we consider equivarint cobordism classes $[M_1]$ and $[M_2]$ in $\Zc_{2n}^U(T^n)$, represented by omnioriented quasitoric manifolds over quotient polytopes $P_1$ and $P_2$, respectively. It then suffices to construct a third such manifold $M$ such that $[M] = [M_1] + [M_2]$. The connect sum in the proof of \cite{BPR1} Theorem 5.9 can be used to define a omnioriented quasitoric manifold $M = M_1\#M_2$ over the connect sum $P^n = P_1\# P_2$. Or we can just use \cite{Darby} Proposition 8.9 to get a quasitoric pair $(P,\lambda)$ and the corresponding omnioriented quasitoric manifold $M$ such that $\varphi_n^U(M) =\varphi_n^U(M_1) + \varphi_n^U(M_2)$. Since $\varphi_n^U$ is monomorphic by Corollary \ref{Cor-Darby}, the equivariant bordism class $[M] = [M_1] + [M_2]\in \Zc_{2n}^U(T^n)$, so the quasitoric manifold $M$ is what we need.  
\end{proof}

\end{document}